\documentclass[12pt]{amsart}
\usepackage{amssymb}
\usepackage{amsthm}
\usepackage{tensor}
\theoremstyle{theorem}
\newtheorem{thm}{Theorem}[section]
\newtheorem{lem}[thm]{Lemma}

\newtheorem{prop}[thm]{Proposition}
\theoremstyle{remark}
\newtheorem{rmk}[thm]{Remark}
\newtheorem{ex}[thm]{Example}
\theoremstyle{definition}
\newtheorem{dfn}[thm]{Definition}

\usepackage{amsmath}

\begin{document}

\title[Cuntz-Pimsner Algebras of Tensor Products]
{Cuntz-Pimsner Algebras Associated to Tensor
Products of $C^*$-Correspondences}
\author{Adam Morgan}
\address{School of Mathematical and Statistical Sciences
\\Arizona State University
\\Tempe, Arizona 85287}
\email{anmorgan@asu.edu}
\date{\today}
\maketitle

\begin{abstract}

Given two $C^*$-correspondences $X$ and $Y$ over $C^*$-algebras $A$
and $B$, we show  that (under mild hypotheses) the Cuntz-Pimsner
algebra $\mathcal O_{X\otimes Y}$ embeds as a certain subalgebra of 
$\mathcal O_X\otimes\mathcal O_Y$ and that this subalgebra can be described
in a natural way in terms of the gauge actions on $\mathcal O_X$ and
$\mathcal O_Y$. We explore implications for graph algebras, crossed products
by $\mathbb Z$, crossed products by completely positive maps, and give a
new proof of a result of Kaliszewski and Quigg related to coactions on 
correspondences.
\end{abstract}

\section{Introduction}
In what follows, we will attempt to describe the Cuntz-Pimsner
algebra of an external tensor product $X\otimes Y$ of correspondences
in terms of the Cuntz-Pimsner algebras $\mathcal O_X$ and $\mathcal O_Y$.
In particular we will show that, under suitable conditions, $\mathcal O
_{X\otimes Y}$ is isomorphic to a certain subalgebra $\mathcal O_X\otimes_
\mathbb T\mathcal O_Y$ of $\mathcal O_X\otimes \mathcal O_Y$. We call this
subalgebra the \emph{$\mathbb T$-balanced tensor product} because it has
the
property that $\gamma^X_z(x)\otimes y=x\otimes\gamma^Y_z(y)$ for all $z\in
\mathbb T$, $x\in\mathcal O_X$ and $y\in\mathcal O_Y$, where $\gamma^X$ and
$\gamma^Y$ are the gauge actions on $\mathcal O_X$ and $\mathcal O_Y$.
This idea is inspired by a result of Kumjian in \cite{Kum} where it is shown
that for a cartesian product $E\times F=(E^0\times F^0,E^1\times F^1,r_E
\times r_F,s_E\times s_F)$ of two graphs, $C^*(E\times F)\cong C^*(E)
\otimes_\mathbb T C^*(F)$ where the balancing is over the gauge action of
the two graphs. Kumjian's proof uses the groupoid model of graph algebras
and is therefore independent of our main result. However, we will be able
to recover Kumjian's result for row finite graphs with no sources by 
considering the $C^*$-correspondence model
of a graph algebra. 

After proving our main result we will explore some examples
including a generalization of Kumjian's result
to the setting of topological graphs, implications for crossed products by
$\mathbb Z$, crossed product by a completely positive map, and we will
give a new proof of a theorem of Kaliszewski and Quigg which was used
in \cite{Quigg2} to study Cuntz-Pimsner algebras of coaction crossed
products.

In what follows we will assume all
tensor products are minimal (spatial) unless otherwise stated.
\section{Preliminaries}
Here we will review the material we will need on correspondences,
Cuntz-Pimsner algebras, actions and coactions.
A good general reference to the subject of Hilbert $C^*$-modules
and $C^*$-correspondences is \cite{Lance}. There is also an overview
given in \cite{Blackadar}.
For Cuntz-Pimsner algebras, we refer the reader to Katsura \cite{Katsura1} and to the overview given in
given in \cite{GraphAlgebraBook}. For actions and coactions, we refer
to appendix A of \cite{enchilada} as a general reference, but since
we will only be concerned with coactions of discrete groups we will also
use a lot of facts from \cite{Quigg}. We will begin
by reviewing the basics of $C^*$-correspondences.
\subsection{Correspondences}

Suppose $A$ is a $C^*$-algebra and $X$ is a right $A$-module.
By an \emph{$A$-valued inner product} on $X$ we shall mean
a map 
$$X\times X\ni(x,y)\mapsto\langle x,y\rangle_A\in A$$
which is linear in the second variable and such that
\begin{enumerate}
\item
$\langle x,x\rangle_A\geq 0$ for all $x\in X$ with equality only when $x=0$
\item
$\langle x, y\rangle^*_A=\langle y,x\rangle_A$ for all $x,y\in X$
\item
$\langle x,y\cdot a\rangle_A=\langle x,y\rangle_A\cdot a$ for all $x,y\in X$ 
and $a\in A$.
\end{enumerate}
Note that this implies that $\langle\cdot,\cdot\rangle_A$ is $A$-linear
in the second variable and conjugate $A$-linear in the first variable.
One can prove a version of the Cauchy Schwartz inequality for such $X$
which implies that we can define the following norm on $X$:
$$\|x\|_A:=\|\langle x,x\rangle_A\|^{\frac 1 2}$$

If $X$ is a right $A$-module with an $A$-valued inner product,
$X$ is called a \emph{right Hilbert $A$-module} if it is complete
under the norm $\|\cdot\|_A$ defined above.
Note that if $A=\mathbb C$ then $X$ is just a Hilbert space and
we can think of general Hilbert modules as Hilbert spaces whose scalars
are elements of some $C^*$-algebra $A$.

Also note that we can make $A$ itself into a right Hilbert $A$-module
by letting the right action of $A$ be given by multiplication in $A$ and
$A$-valued inner product given by $\langle a,b\rangle_A=a^*b$. We call this
the \emph{trivial} Hilbert $A$-module and denote it by $A_A$.

Let $A$ be a $C^*$-algebra and let 
$X$ be a right Hilbert $A$-module. If $T:X\to X$ is an
$A$-module homomorphism, then we call $T$ \emph{adjointable}
if there is an $A$-module homomorphism $T^*$ (called the 
\emph{adjoint} of $T$) such that
$$\langle Tx,y\rangle_A=\langle x, Ty\rangle_A$$
for all $x,y\in X$. The operator norm makes the set of all adjointable
operators on $X$ into a $C^*$-algebra which we denote by $\mathcal L(X)$.

Given $C^*$-algebras $A$ and $B$, an \emph{$A-B$-correspondence}
is right Hilbert $B$-module $X$ together with a map $\phi:A\to\mathcal L(X)$
which is called a \emph{left action of $A$ by adjointable operators}.
For $a\in A$ and $x\in X$, we will write $a\cdot x$ for $\phi(a)(x)$.
If $A=B$ we call this a \emph{correspondence over $A$} (or $B$).
We call the left-action \emph{injective} if $\phi$ is injective
and \emph{non-degenerate} if $\phi(A)X=X$.
We will sometimes write $_AX_B$ to indicate that $X$ in an $A-B$
correspondence.
Before we continue, we will give a few examples of correspondences.
\begin{ex}[Example 8.6 in \cite{GraphAlgebraBook}]
Let $A$ be a $C^*$ algebra and let $\alpha$ be an endomorphism of $A$.
We can make the trivial module $A_A$ into a correspondence over $A$ by defining
$\phi(a)(x)=\alpha(a)x$.
\end{ex}
\begin{ex}[Example 8.7 in \cite{GraphAlgebraBook}]\label{GraphCor}
Let $E=\{E^0,E^1,r,s\}$ be a directed graph (in the sense of 
\cite{GraphAlgebraBook}). Consider the vector space $c_c(E^1)$
of finitely supported functions on $E^1$. We can define a right
action of $c_0(E^0)$ and a $c_0(E^0)$-valued inner product as follows:
\begin{align*}
(x\cdot a)(e)&=x(e)a(s(e))\\
\langle x,y\rangle_{c_0(E^0)}(v)&=\sum_{\{e\in E^1:s(e)=v\}}\overline{x(e)}y(e)
\end{align*}
We can use this inner product to complete $c_c(E^1)$ into a right Hilbert
module $X(E)$. We can define a left action $\phi:c_0(E^0)\to \mathcal L(X(E))$
as follows
$$\phi(a)(x)(e)=a(r(e))x(e)$$
This makes $X(E)$ into a correspondence which is referred to as the 
\emph{graph correspondence} of $E$.
\end{ex}

This example has a natural generalization:
\begin{ex}[Topological graphs]
A \emph{topological graph} is a quadruple $E=\{E^0,E^1,r,s\}$
where $E^0$ and $E^1$ are locally compact Hausdorff spaces,
$r:E^1\to E^0$ is a continuous function, and $s:E^1\to E^0$
is a local homeomorphism. Let $A:=C_0(E^0)$. We can define
left- and right-actions of $A$ and an $A$-valued inner product on $C_c(E^1)$
similarly to the way we did for ordinary graphs (\cite{GraphAlgebraBook}
Chapter 9): For $a\in A$ and $x,y\in C_c(E^1)$, let
\begin{align*}
(a\cdot x)(e)&:=a(r(e))x(e)\\
(x\cdot a)(e)&:=x(e)a(s(e))\\
\langle x,y\rangle_A(v)&:=\sum_{\{e\in E^1:s(e)=v\}}\overline{x(e)}
y(e)
\end{align*}
We denote the completion of $C_c(E^1)$ under this inner product by
$X(E)$. It can be shown (\cite{GraphAlgebraBook} p.80-81) that the left action
is injective if and only if $r$ has dense range and the left action is
implemented by compacts if and only if $r$ is proper.
\end{ex}
Let $X$ be a Hilbert module over a $C^*$-algebra $A$. Then for any 
$x,y\in X$ the map $\Theta_{x,y}:z\mapsto x\cdot\langle y,z\rangle_A$
is an adjointable operator called a \emph{rank-one} operator.
It can be shown that the closed span of the rank-one operators forms an
ideal $\mathcal K(X)$ in $\mathcal L(X)$ which is referred to as the set
of \emph{compact operators}. If $X$ is an $A-B$-correspondence, we say
that the left action $\phi$ of $A$ is \emph{implemented by compacts}
if $\phi(A)\subseteq \mathcal K(X)$.

There are two types of tensor products which are usually defined on
correspondences, an ``internal'' tensor product and an ``external''
tensor product. These are defined as follows (see chapter 4 of \cite{Lance}
for more detail):
Let $_AX_B$ and $_CY_D$ be correspondences and let $\Phi:B\to C$ be a completely
positive map (see \cite{Blackadar} for the basics of completely positive
maps). Let $X\odot_\Phi Y$ be the quotient of the algebraic tensor product
$X\odot Y$ by the subspace spanned by
$$\{x\cdot b\otimes y-x\otimes\Phi(b)\cdot y:x\in X, y\in Y, b\in B\}$$
This is a right $D$-module with right action given by $(x\otimes y)\cdot
d=(x\otimes y\cdot d)$. We define a $D$-valued inner product as follows:
$$\langle x\otimes y, x'\otimes y'\rangle=\langle y,\Phi\big(\langle x,x'\rangle_B
\big)\cdot y'\rangle_D$$
We refer to the completion of $X\odot_\Phi Y$ with respect to this inner
product as the \emph{internal tensor product} of $X$ and $Y$ and it is
denoted by $X\otimes_\Phi Y$. If $\phi_A:A\to\mathcal L(X)$ is the left
action on $X$, then we can define a left action of $A$ on $X\otimes_\Phi
Y$ by $\phi(a)(x\otimes y)=\big(\phi_A(a)x\big)\otimes y$. This makes
$X\otimes_\Phi Y$ an $A-D$ correspondence.
In many situations we will have $B=C$ and
$\Phi=id_B$. In this case we will write the associated internal tensor product
as $X\otimes_B Y$.
\begin{ex}\label{CPCor}
If $\Phi:A\to B$ is a completely positive map between two $C^*$-algebras,
then we define the \emph{correspondence associated to $\Phi$} to be
the correspondence $X_\Phi:=\tensor[_A]A{_A}\otimes_\Phi\tensor[_B]B{_B}$
where $\tensor[_A]A{_A}$ and
$\tensor[_B]B{_B}$ are the standard correspondences.
\end{ex}
Let $_A X_B$ and $_CY_D$ be correspondences. We can define a
right action of $B\otimes D$ and a $B\otimes D$-valued inner
product on the algebraic tensor product $X\odot Y$ as follows
\begin{align*}
(x\otimes y)\cdot(a\otimes b)&=(x\cdot a)\otimes (y\cdot b)\\
\langle x\otimes y,x'\otimes y'\rangle&=\langle x,x'\rangle_B
\otimes\langle y,y'\rangle_D
\end{align*}
The completion of $X\odot Y$ with respect to this inner product
is called the \emph{external tensor product} of $X$ and $Y$ which
we will denote simply by $X\otimes Y$.
We can define a left action of $A\otimes C$ as follows:
$\phi(a\otimes c)(x\otimes y)=\big(\phi_A(a)x\big)\otimes\big(\phi_C(c)y\big)$.
This makes $X\otimes Y$ an $A\otimes C-B\otimes D$ correspondence.

If $\Phi:A\to C$ and $\Phi':B\to D$ are completely positive maps,
then $\Phi\otimes\Phi':A\otimes B\to C\otimes D$ will be a completely 
positive map as well. In fact:
\begin{lem}
Let $_AX_{A'}$, $_BY_{B'}$, $_CZ_{C'}$, and $_DW_{D'}$ be correspondences
and let $\Phi:A'\to C$ and $\Phi':B'\to D$ be completely positive maps.
Then
$$(X\otimes_\Phi Z)\otimes(Y\otimes_{\Phi'}W)\cong(X\otimes Y)\otimes_{\Phi\otimes\Phi'}
(Z\otimes W)$$
\end{lem}
\begin{proof}
Note that the left hand side is the completion of 
$X\odot Z\odot Y\odot W$ under a ceratin pre-inner product and
the right hand side is the completion of $X\odot Y\odot Z\odot W$ under
a certain pre-inner product. We can show that the linear map
\begin{align*}
\sigma_{23}:X\odot Z\odot Y\odot Z&\to X\odot Y\odot Z\odot W\\
x\otimes z\otimes y\otimes z&\mapsto x\otimes y\otimes z\otimes w
\end{align*}
extends to a correspondence isomorphism
$$(X\otimes_\Phi Z)\otimes(Y\otimes_{\Phi'}W)
\to(X\otimes Y)\otimes_{\Phi\otimes\Phi'}(Z\otimes W)$$
by showing that $\sigma_{23}$ preserves the pre-inner products. By linearity
it suffices to show this for elementary tensors. Let
$\langle\cdot,\cdot\rangle_1$ denote the pre-inner product which gives
rise to $(X\otimes_\Phi Z)\otimes(Y\otimes_{\Phi'}W)$.
Let $\langle\cdot,\cdot\rangle_2$ denote the pre-inner product which gives
rise to $(X\otimes Y)\otimes_{\Phi\otimes\Phi'}(Z\otimes W)$. Then
\begin{align*}
\big\langle\sigma_{23}(x\otimes z\otimes y\otimes w),&\sigma_{23}
(x\otimes z\otimes y\otimes w)\big\rangle_2\\
&=\langle x\otimes y\otimes z\otimes w,x'\otimes y'\otimes z'\otimes w'\rangle_2\\
&=\Big\langle z\otimes w,(\Phi\otimes\Phi')\big(\langle x\otimes y,x'\otimes
y'\rangle\big)(z'\otimes w')\Big\rangle\\
&=\Big\langle z\otimes w,(\Phi\otimes \Phi')\big(\langle x,x'\rangle\otimes\langle
y,y'\rangle\big)(z'\otimes w')\Big\rangle\\
&=\Big\langle z\otimes w,\big(\Phi\big(\langle x,x'\rangle\big)z'\big)\otimes
\big(\Phi'\big(\langle y,y'\rangle\big)w'\big)\Big\rangle\\
&=\Big\langle z,\Phi\big(\langle x,x'\rangle\big)z'\Big\rangle\otimes\Big\langle
w,\Phi'\big(\langle y,y'\rangle\big)w'\Big\rangle\\
&=\langle x\otimes z\otimes y\otimes w, x'\otimes z'\otimes y'\otimes z'\rangle_1
\end{align*}
Thus $\sigma_{23}$ extends to an isomorphism giving us 
$$(X\otimes_\Phi Z)\otimes(Y\otimes_{\Phi'}W)\cong (X\otimes Y)\otimes_{\Phi\otimes
\Phi'}(W\otimes Z)$$
\end{proof}
\begin{ex}\label{TensorProductOfCPMap}
Let $\Phi:A\to C$ and $\Phi':B\to D$ be completely positive maps.
Since $\tensor[_A]A{_A}\otimes\tensor[_B]B{_B}=\tensor[_{A\otimes B}]{(A\otimes
B)}{_{A\otimes B}}$ and 
$\tensor[_C]C{_C}\otimes\tensor[_D]D{_D}={\tensor[_{C\otimes D}]{(C\otimes
D)}{_{C\otimes D}}}$
We have:
\begin{align*}
X_{\Phi\otimes\Phi'}&=\left(\tensor[_{A\otimes B}]{(A\otimes
B)}{_{A\otimes B}}\right)\otimes_{\Phi\otimes\Phi'}\left({\tensor
[_{C\otimes D}]{(C\otimes D)}{_{C\otimes D}}}\right)\\
&=\big(\tensor[_A]A{_A}\otimes\tensor[_B]B{_B}\big)\otimes_{\Phi\otimes\Phi'}
\big(\tensor[_C]C{_C}\otimes\tensor[_D]D{_D}\big)
\end{align*}
Applying the preceding lemma gives:
\begin{align*}
&\cong\big(\tensor[_A]A{_A}\otimes_\Phi\tensor[_C]C{_C}\big)\otimes\big(
\tensor[_B]B{_B}\otimes_{\Phi'}\tensor[_D]D{_D}\big)\\
&=X_\Phi\otimes X_{\Phi'}
\end{align*}
Thus $X_{\Phi\otimes\Phi'}\cong X_\Phi\otimes X_{\Phi'}$.
\end{ex}

The following facts (see Chapter 4 of \cite{Lance}) will be useful
in proving our main result:
\begin{lem}\label{TensorOfCompacts}
Let $X$and $Y$ be $C^*$-correspondences over $C^*$-algebras $A$ 
and $B$
respectively. Then $\mathcal K(X\otimes^{ext}Y)\cong\mathcal K(X)\otimes
\mathcal K(Y)$ via the map $\kappa$ which takes 
$S\otimes T\in\mathcal K(X)\otimes 
\mathcal K(Y)$ to the linear map $x\otimes y\mapsto Sx\otimes Ty$. 
Further, if the left actions of $A$ and $B$ are injective
and implemented by compacts then so is the left action of $A\otimes B$ on
$X\otimes^{ext}Y$. 
\end{lem}
Before we end this section, we will give another example of an external tensor
product:

\begin{ex}\label{TensorTopGraph}
Let $E=\{E^0,E^1,r,s\}$ and $F:=\{E^0,E^1,r',s'\}$ 
be topological graphs. Define 
$$E\times F:=\{E^0\times F^0,E^1\times F^1,r\times r',s\times s'\}$$
(i.e. the topological analog of the product graph in \cite{Kum}).
Since the product of two continuous maps is continuous and the product
of two local homeomorphisms is a local homeomorphism, $E\times F$ is 
a topological graph. Let $\rho:C_0(E^0)\otimes C_0(F^0)\to C_0(E^0\times
F^0)$ and $\sigma:C_c(E^1)\otimes C_c(F^1)\to C_c(E^1\times F^1)$
be the standard isomorphisms. Note the following:
\begin{align*}
\sigma\big((a\otimes b)\cdot(x\otimes y)\big)(e\times f)&=\phi(a\cdot x
\otimes b\cdot y)(e\times f)\\
&=a\big(r(e)\big)x(e)b\big(r'(f)\big)y(f)\\
&=a\big(r(e)\big)b\big(r'(f)\big)x(e)y(f)\\
&=\rho(a\otimes b)\big(r(e)\times r'(f)\big)\sigma(x\otimes y)(e\times f)\\
\end{align*}
Similarly, 
$$\sigma\big((x\otimes y)\cdot(a\otimes b)\big)(e\times f)=
\sigma(x\otimes y)(e\times f)\rho(a\otimes b)\big(s(e)\times s'(f)\big)$$
Let $A=C_0(E^0)$ and $B=C_0(F^0)$. If $\langle\cdot,\cdot\rangle_1$ is the
$A\otimes B$-valued
tensor product associated to the external tensor product $X(E)\otimes X(F)$
and $\langle\cdot,\cdot\rangle_2$ is the $C_0(E\times F)$ valued inner product
associated to $X(E\times F)$ then:
\begin{align*}
\rho\big(\langle x\otimes y,x'\otimes y'\rangle_1\big)(v\times
w)&=\rho\big(\langle x,x'\rangle_A\otimes\langle y,y'\rangle_B\big)
(v\times w)\\
&=\langle x,x'\rangle_A(v)\langle y,y'\rangle_B(w)\\
&=\sum_{s(e)=v, s'(f)=w}\overline {x(e)}x'(e)\overline{y(f)}y'(f)\\
&=\sum_{s(e)=v, s'(f)=w}\overline {x(e)y(f)}x'(e)y'(f)\\
&=\sum_{s(e)\times s'(f)=v\times w}\overline {\sigma(x\otimes y)(e\times
f)}
\sigma(x'\otimes y')(e\times f)\\
&=\langle\sigma(x\otimes y),\sigma(x'\otimes y')\rangle_2(v\times w)
\end{align*}
Extending linearly and continuously, we see that $(\rho,\sigma)$ gives
an isomorphism of correspondences: $X(E\times F)\cong X(E)\otimes X(F)$.
\end{ex}
\subsection{Cuntz-Pimsner Algebras}
In order to describe the 
Cuntz-Pimsner algebra of a correspondence, we will
first need to discuss representations of correspondences.
Given a correspondence $X$ over a $C^*$-algebra $A$, a \emph{Toeplitz
Representation} of $X$ is a $C^*$ algebra $B$ is a pair $(\psi,\pi)$
where $\psi:X\to B$ is a linear map and $\pi:A\to B$ is a $*$-homomorphism
such that
\begin{enumerate}
\item
$\psi(a\cdot x)=\pi(a)\psi(x)$ for all $a\in A$ and all $x\in X$
\item
$\psi(x\cdot a)=\psi(x)\pi(a)$ for all $a\in A$ and all $x\in X$
\item
$\pi\big(\langle x,y_A\rangle\big)=\psi(x)^*\psi(y)$ for all $x,y\in X$
\end{enumerate}
(Note that 2 is actually implied by 3 but we include it for clarity.)
We will write $C^*(\psi,\pi)$ for the $C^*$-subalgebra of $B$ generated
by the images of $\psi$ and $\pi$ in $B$.
It can be shown that there is a unique (up to isomorphism)
$C^*$-algebra $\mathcal T_X$, called
the \emph{Toeplitz algebra} of $X$, which is generated by a representation
$(i_X,i_A)$ which is ``universal'' in the sense that for any representation
$(\psi,\pi)$ of $X$ in any $C^*$-algebra $B$, there is a unique $*$-homomorphism
$\psi\times\pi:\mathcal T_X\to B$ such that $\psi=(\psi\times\pi)\circ i_A$
and $\pi=(\psi\times\pi)\circ i_X$.
Let $X^{\otimes n}$ denote the n-fold internal tensor product of $X$
with itself by convention we let $X^{\otimes 0}=A$. Given a Toeplitz
representation $(\psi,\pi)$ of $X$ in $B$, we define a map 
$\psi^n:X^{\otimes n}\to B$ for each $n\in \mathbb N$ as follows:
We let $\psi^0=\pi$ and $\psi^1=\psi$ and then for each $n>1$ we
set $\psi^n(x\otimes y)=\psi(x)\psi^{n-1}(y)$ where $x\in X$ and $y\in 
X^{\otimes n-1}$.

Proposition 2.7 of \cite{Katsura1} states the following:
Let $X$ be correspondence over $A$ and let $(\psi,\pi)$ be
a Toeplitz representation of $X$. Then
$$C^*(\psi,\pi)=\overline{span}\{\psi^n(x)\psi^m(y)^*:x\in X^{\otimes n},
y\in X^{\otimes m}\}$$

From Lemma 2.4 of \cite{Katsura1} we get the following:
Let $(\psi,\pi)$ be a Toeplitz representation of $X$ in $B$.
For each $n\in\mathbb N$ there is a homomorphism $\psi^{(n)}:
\mathcal K(X^{\otimes n})\to B$ such that:
\begin{enumerate}
\item
$\pi(a)\psi^{(n)}(k)=\psi^{(n)}(\phi(a)k)$ for all $a\in A$ and all $k\in
\mathcal K(X^{\otimes n})$.
\item
$\psi^{(n)}(k)\psi(x)=\psi(kx)$ for all $x\in X$ and all $k\in \mathcal 
K(X^{\otimes n})$
\end{enumerate}
Let $X$ be a correspondence over a $C^*$-algebra $A$. We define
the \emph{Katsura ideal} of $A$ to be the ideal:
$$J_X=\{a\in A:\phi(a)\in \mathcal K(X) \text{ and }ab=0\text{ for all
}b\in ker(\phi)\}$$
Where $\phi$ is the left action. This is often written more compactly
as $J_X=\phi^{-1}\big(\mathcal K(X)\big)\cap\big(ker(\phi)\big)^\perp$. 
In many cases of interest, one can consider only
correspondences whose left actions are injective and implemented by compacts.
In this case we have $J_X=A$.

The Katsura ideal is also sometimes described as the largest ideal of $A$
which maps injectively onto the compacts. This is made precise by the following
proposition (see \cite{Katsura1}):
\begin{prop}\label{CharacterizationOfKatIdeal}
Suppose $X$ is a correspondence over a $C^*$-algebra $A$, $\phi$ is the 
left action map, and $I$ is an ideal of $A$ which is mapped injectively
into $\mathcal K(X)$ by $\phi$. Then $I\subseteq J_X$.
\end{prop}

We are now ready to define the Cuntz-Pimsner algebra $\mathcal O_X$.
A Toeplitz representation is said to be \emph{Cuntz-Pimsner covariant}
if $\psi^{(1)}\big(\phi(a)\big)=\pi(a)$ for all $a\in J_X$. The 
\emph{Cuntz-Pimsner algebra} $\mathcal O_X$ is the quotient of 
$\mathcal T_X$ by the ideal generated by 
$$\{i_X^{(1)}\big(\phi(a)\big)-i_A(a):a\in J_X\}$$
Letting $q:\mathcal T_X\to\mathcal O_X$ denote the quotient map,
it can be shown that $\mathcal O_X$ is generated by the Cuntz-Pimsner
covariant representation $(k_X,k_A)={(q\circ i_X, q\circ i_A)}$ and
that this representation is universal for Cuntz-Pimsner covariant
representations: if $(\psi,\pi)$ is a Cuntz-Pimsner covariant
representation of $X$ in $B$, then there is a $*$-homomorphism
$\psi\times\pi:\mathcal O_X\to B$ such that $\psi=(\psi\times\pi)\circ
k_X$ and $\pi=(\psi\times\pi)\circ k_A$.

\begin{ex}
If $E$ is a directed graph and $X(E)$ is the graph correspondence of 
Example \ref{GraphCor}, then one can show that $\mathcal O_{X(E)}\cong C^*(E)$
where $C^*(E)$ is the $C^*$-algebra of the graph (see \cite{GraphAlgebraBook}
for details). For this reason, if $F$ is a topological graph, the graph algebra
of $F$ is defined to be $\mathcal O_{X(F)}$. 
\end{ex}
One of the most important results about Cuntz-Pimsner algebras is the
so-called ``gauge-invariant uniqueness theorem''. In order to state this
theorem we need the following definition:
Let $(\psi,\pi)$ be a Toeplitz representation of a correspondence $X$.
Then we say that $C^*(\psi,\pi)$ \emph{admits a gauge action} if there
is an action $\gamma$ of $\mathbb T$ on $C^*(\psi,\pi)$ such that:
\begin{enumerate}
\item
$\gamma_z\big(\pi(a)\big)=\pi(a)$ for all $z\in \mathbb T$ and all $a\in
A$
\item
$\gamma_z\big(\psi(x)\big)=z\psi(x)$ for all $z\in\mathbb T$ and all
$x\in X$
\end{enumerate}
When such an action exists, it is unique.

\begin{thm}[``Gauge Invariant Uniqueness Theorem'' (6.4 in \cite{Katsura1})]
Let $X$ be a correspondence over $A$ and let $(\psi,\pi)$ be a Cuntz-Pimsner
covariant representation of $X$. Then the $*$-homomorphism
$\psi\times\pi:\mathcal O_X\to C^*(\psi,\pi)$ is an isomorphism if and
only if $(\psi,\pi)$ is injective and admits a gauge action.
\end{thm}

\subsection{Actions, Coactions and Gradings}
We will make use of the relationship between the gauge action of 
$\mathbb T$ and the natural $\mathbb Z$-grading of Cuntz-Pimsner algebras.
This relationship comes from the duality between actions of an abelian group
and coactions of the dual group. We will briefly recall the basics of actions
and coactions here.
By an \emph{action} of a locally compact group $G$ on a $C^*$-algebra
$A$, we shall mean
a strongly continuous group homomorphism $\alpha:G\to Aut(A)$. We will refer
to the triple $(A,G,\alpha)$ as a \emph{``$C^*$-dynamical system''}. 
For $s\in G$ we
will write $\alpha_s$ for the automorphism $\alpha(s)$.
Let $G$ be a discrete group. Then the group $C^*$-algebra $C^*(G)$
is generated by a unitary representation $G\ni s\mapsto u_s\in \mathcal UC^*(G)$.
We will abuse notation
and write $s$ for $u_s$. The function $\delta_G:C^*(G)\to C^*(G)\otimes
C^*(G)$ defined by $s\mapsto s\otimes s$ for all $s\in G$ is called the 
\emph{comultiplication map} on $C^*(G)$. There is a comultiplication
map defined for any locally compact group, but in general it will map
into $M(C^*(G)\otimes C^*(G))$. See the appendix of \cite{enchilada}
for more details.

A \emph{coaction} of a group $G$ on a $C^*$-algebra $A$
is a non-degenerate, injective homomorphism $\delta:A\to
M(A\otimes C^*(G))$ such that
\begin{enumerate}
\item
$\delta(A)(1\otimes C^*(G))\subseteq A\otimes C^*(G)$
\item
$(\delta\otimes id_G)\circ\delta=(id_A\otimes\delta_G)\circ\delta$ where
both sides are viewed as maps $A\to M(A\otimes C^*(G)\otimes C^*(G))$
\end{enumerate}
A coaction is called \emph{nondegenerate} if 
the closed linear span of $\delta(A)(1\otimes C^*(G))$
is dense in $A\otimes C^*(G)$. We will also refer to the
triple $(A,G,\delta)$
as a coaction.
Note that if $G$ is discrete then $C^*(G)$ is unital (with unit $u_e$)
and thus $\delta(A)\subseteq\delta(A)(1\otimes C^*(G))$ so by condition
1, $\delta$ maps into $A\otimes C^*(G)$. In fact, the results
of \cite{Quigg} tell us that (nondegenerate) 
coactions of discrete groups correspond to
group gradings. We will discuss this more formally in a moment, but
first we will state precisely what we mean by a grading of a $C^*$-algebra.
Let $G$ be a discrete group and $A$ a $C^*$-algebra. By a 
\emph{$G$-grading} of $A$ we shall mean a collection $\{A_s\}_{s\in G}$
of linearly independent closed subspaces of $A$ such that the following hold:
\begin{enumerate}
\item
$A_sA_t\subseteq A_{st}$ for all $s,t\in G$
\item
$A_s^*=A_{s^{-1}}$
\item
$A=\overline{span}_{s\in G}(A_s)$
\end{enumerate}
We will say a grading is \emph{full} if we have $\overline{A_sA_t}
=A_{st}$ for all $s,t\in G$.
We summarize Lemma 1.3 and Lemma 1.5 of \cite{Quigg} as follows:
\begin{lem} \label{grading}
A nondegenerate coaction of a discrete group
$G$ on a $C^*$-algebra $A$ gives a $G$-grading of $A$. 
Specifically, we let $A_s=\{a\in A: \delta(a)=a\otimes s\}$ for each
$s\in G$.
\end{lem}
\begin{rmk}[Example A.23 of \cite{enchilada}]
If $G$ is abelian, then for every coaction of $G$ corresponds to
an action of the dual group $\widehat G$ and vice versa. To see this,
we first identify $C^*(G)$ and $C_0(\widehat G)$ using the abstract
Fourier transform: $\mathcal F(x)(\chi)=\chi(x)$. We also recall that,
in this situation, condition 1 from the definition of coactions is 
equivalent to $\delta$ taking values in $C_b(\widehat G,A)\in
M(A\otimes C^*(G))$ (see \cite{enchilada}
for details). With this in mind, if $(A,G,\delta)$ is a coaction,
then we define an action $\alpha^\delta$ of $\widehat G$ by setting
$\alpha^\delta_\chi(a)=\delta(a)(\chi)$.  Conversely, given an action
$\alpha$ of $\widehat G$ on a $C^*$-algebra $A$, we define a coaction
$\delta^\alpha$ by letting $\delta^\alpha(a)(\chi)=\alpha_\chi(a)$.
\end{rmk}
Let $(A,G,\alpha)$ be a $C^*$-dynamical system with $G$ compact and
abelian and let $(A,\widehat G,\delta^\alpha)$ be the associated
coaction of the dual group $\widehat G$.
Since $G$ is compact, $\widehat G$ will be discrete
and therefore $\delta^\alpha$ will give a $\widehat G$-grading
of $A$ as in Lemma \ref{grading}.
Identifying $A\otimes C^*(\widehat G)$ with $A\otimes C(G)$ and this
with $C(G,A)$, the elementary tensor $a\otimes u_\chi$ corresponds
to the map $s\mapsto \chi(s)a$. Therefore if $a\in A_\chi$ then
$\alpha_s(a)=\delta^\alpha(a)(s)=\chi(s)a$
and so the sets $A_\chi$ can be thought of equivalently as
$$A_\chi=\{a\in A:\alpha_s(a)=\chi(s)a\}$$
Thus each $A_\chi$ coincides with the so-called \emph{spectral 
subspace} associated to $\chi$.

Just as every action of a compact abelian group determines a grading
by the dual group, every grading
by a discrete abelian group determines an action of its dual group.
To see this, just note that a $G$-grading of $A$ makes $A$
into a a Fell bundle over $G$ and as in \cite{Quigg} we get a coaction
of $G$ associated to this Fell bundle which corresponds to an action
of $\widehat G$:
\begin{prop}\label{ActionFromGrading}
Let $A$ be a $C^*$-algebra and suppose $\{A_s\}_{s\in G}$ is a $G$-grading
of $A$. Let $\chi\in \widehat G$. For each $s\in G$ and each $a\in A_s$ 
define $\alpha_\chi(a)=\chi(s)a$. The maps $\alpha_\chi$ extend to automorphisms
of $A$ such that $\alpha:\chi\mapsto\alpha_\chi$ is an action of $\widehat
G$ on $A$.
\end{prop}

\section{Tensor Products Balanced Over Group Actions or Group Gradings}
We wish to show that the Cuntz-Pimsner algebra of an external tensor product
$X\otimes Y$
of correspondences is isomorphic to a certain subalgebra $\mathcal O_X
\otimes_\mathbb T\mathcal O_Y$ of the tensor product
$\mathcal O_X\otimes \mathcal O_Y$. This subalgebra is called the 
\emph{$\mathbb T$-balanced tensor product} of $\mathcal O_X$ and
$\mathcal O_Y$. Following \cite{Kum} we define the general construction
as follows:
\begin{dfn}
Let $G$ be a compact abelian group, and let $(A,G,\alpha)$ 
and $(B,G,\beta)$ 
be $C^*$-dynamical systems.
We define the \emph{$G$-balanced tensor product} $A\otimes_G B$ to be
the fixed point algebra $(A\otimes B)^\lambda$ where $\lambda:G\to A\otimes
B$ is the action characterized by $\lambda_s(a\otimes b)=\alpha_s(a)\otimes
\beta_{s^{-1}}(b)$. 
\end{dfn}
\begin{prop}
If $a\otimes b\in A\otimes_G B$ then $\alpha_s(a)\otimes b, a\otimes
\beta_s(b)\in A\otimes_G B$ for all $s\in G$ and $\alpha_s(a)\otimes b
=a\otimes\beta_s(b)$. 
\end{prop}
\begin{proof} To show that $\alpha_s(a)\otimes b\in A\otimes_G B$ 
note that for any $t\in G$ we have
\begin{align*}
\alpha_t\big(\alpha_s(a)\big)\otimes \beta_{t^{-1}}(b)&=\alpha_s
\big(\alpha_t(a)\big)\otimes\beta_{t^{-1}}(b)\\
&=(\alpha_s\otimes id_B)(\alpha_t(a)\otimes \beta_{t^{-1}}(b))\\
&=(\alpha_s\otimes id_B)(a\otimes b)\\
&=\alpha_s(a)\otimes b
\end{align*}
showing that $a\otimes \beta_s(b)\in A\otimes_G B$ is similar.
Now that this has been established, the equality follows easily:
\begin{align*}
\alpha_s(a)\otimes b&=\alpha_{s^{-1}}\big(\alpha_s(a)\big)\otimes
\beta_s(b)\\
&=a\otimes \beta_s(b)
\end{align*}
\end{proof}
Thus the actions $\alpha\otimes \iota_A$ and $\iota_B
\otimes\beta$ coincide on $A\otimes_G B$ where $\iota_A$ and $\iota_B$ are
the trivial actions. We will refer to the restriction of $A\otimes\iota_B$
to $A\otimes_G B$ (or equivalently the restriction of $\iota_A\otimes \beta$
to $A\otimes_G B$) as the \emph{balanced action} of $G$ and we will denote it by $\alpha\otimes_G\beta$.

The main result of this paper can be stated roughly as 
$$\mathcal O_{X\otimes Y}\cong\mathcal O_X\otimes_{\mathbb T}\mathcal O_Y$$
for suitable $X$ and $Y$ where we are balancing over the gauge actions on
$\mathcal O_X$ and $\mathcal O_Y$. This generalizes the following example
from \cite{Kum}:
\begin{ex}
Let $E$ and $F$ be source-free, row-finite discrete graphs and let 
$E\times F$ denote the product graph as in Example \ref{TensorTopGraph}.
Then $C^*(E\times F)\cong C^*(E)\otimes_\mathbb T C^*(F)$.
\end{ex} 


As we noted above, actions of compact abelian groups
correspond to gradings of the dual group. It will be useful to be able
to describe $\mathbb T$-balanced tensor products in terms of the corresponding
$\mathbb Z$-gradings. But first we will need a fact
which follows from the Peter-Weyl theorem (Theorem VII.1.35 of \cite{Gaal})
\begin{lem}\label{Haar}
Let $G$ be a compact abelian group with a normalized Haar measure
and let $\chi$ be a character (i.e. a
continuous homomorphism $G\to\mathbb T$). Then
$$\int_G\chi(s)ds=
\begin{cases}
1 & \text{$\chi$ is the trivial homomorphism}\\
0 & \text{otherwise}
\end{cases}$$
\end{lem}
Now we are ready to describe $G$-balanced tensor products in terms of gradings
of $\widehat G$.
\begin{prop}\label{l1}
Let $(A,G,\alpha)$ and $(B,G,\beta)$ be $C^*$-dynamical systems with $G$
abelian. Then as discussed in the previous section, the coactions 
$\delta^\alpha$ and $\delta^\beta$ give
$\widehat G$-gradings of $A$ and $B$: 
\begin{align*}
A_\chi&=\{a\in A: \alpha_s(a)=\chi(s)a\}\\
B_\chi&=\{b\in B: \beta_s(b)=\chi(s)b\}
\end{align*}
Let $$S_\chi:=\{a\otimes b: a\in A_\chi, b\in B_\chi\}$$
and let $$S:=\bigcup_{\chi\in\widehat G}S_\chi$$
Then $A\otimes_G B=\overline{span}(S)$.
\begin{proof}
First, note that if 
$a\otimes b\in S$, then $a\otimes b\in S_\chi$ for some $\chi$ and
so for all $s\in G$:
\begin{align*}
\lambda_s(a\otimes b)&=\alpha_s(a)\otimes \beta_{s^{-1}}(b)\\
&=\chi(s)a\otimes \chi(s^{-1})b\\
&=\chi(s)\chi(s^{-1})(a\otimes b)\\
&=a\otimes b
\end{align*}
so $a\otimes  b\in A\otimes_G B$ and hence $S\subseteq A
\otimes_G B$. Since $A\otimes_G B$ is a $C^*$-algebra,
we have that $\overline{span}(S)\subseteq A\otimes_G B$.

Now, since $A$ is densely spanned by the $A_\chi$'s and $B$ is densely
spanned by the $B_\chi$'s, the tensor product $A\otimes B$ will be 
densely spanned by elementary tensors $a\otimes b$ where $a\in A_\chi$
and $b\in B_{\chi'}$ for some $\chi,\chi'\in\widehat G$. More precisely,
let 
$$T_{\chi,\chi'}=\{a\otimes b:a\in A_\chi,b\in B_{\chi'}\}$$ and let
$$T=\bigcup_{\chi,\chi'\in\widehat G}T_{\chi,\chi'}$$
Then we have $A\otimes B=\overline{span}(T)$. 
Let $\varepsilon:A\otimes B\to A\otimes_G B$ be the
conditional expectation $c\mapsto\int_G\lambda_s(c)ds$.
Then, since $\varepsilon$ is continuous, linear and surjective, 
$\varepsilon(T)$ densely spans $A\otimes_G B$. Let 
$a\otimes b\in T$, say $a\otimes b\in T_{\chi,\chi'}$. Then
using Lemma \ref{Haar} and the fact that a product of two
characters is a character, we have
\begin{align*}
\varepsilon(a\otimes b)&=\int_G\lambda_s(a\otimes b)ds\\
&=\int_G\big(\chi(s)a\otimes \chi'(s^{-1})b\big)ds\\
&=\left(\int_G\chi(s)\overline{\chi'(s)}dz\right)a\otimes b\\
&=\left(\int_G \left(\chi\overline{\chi'}\right)(s)ds\right)
a\otimes b\\
&=
\begin{cases}
1 & \text{if $\chi\overline{\chi'}$ is trivial}\\
0 & \text{otherwise}
\end{cases}
\end{align*}

But $\chi\overline{\chi'}$ will be trivial if and only if $\chi=\chi'$.
This means that if $\chi\neq\chi'$ (i.e. $a\otimes b\in T\backslash S)$ then
$\varepsilon(a\otimes b)=0$ and if $\chi=\chi'$ (i.e. $a\otimes b\in S$)
Then $\varepsilon(a\otimes b)=a\otimes b$. This implies
that $\varepsilon(T)=S$ and thus, since $T$ densely spans $A\otimes B$, the
linearity and continuity of $\varepsilon$ imply that $S$ densely spans
$\varepsilon(A\otimes B)=A\otimes_G B$.
Therefore $A\otimes_G B=\overline{span}(S)$.
\end{proof}
\end{prop}

Proposition \ref{ActionFromGrading} tells us that the $\widehat G$-grading
of $A\otimes_G B$ just described should give us an action of $G$ on $A\otimes_GB$.
We will now show that this action coincides
exactly with the balanced action:
\begin{prop}
Let $\{S_\chi\}_{s\in \widehat G}$ be the $\widehat G$-grading of $A\otimes_GB$
described in the previous proposition and let $\gamma$ be the action
associated to this grading by Proposition \ref{ActionFromGrading}.
Then $\gamma=\alpha\otimes_G\beta$.
\end{prop}
\begin{proof}
It suffices to check that these maps coincide on each subspace $S_\chi$.
Let $a\otimes b\in S_\chi$. Then for each $s\in\widehat G$ we have
$\gamma_s(a\otimes b)=\chi(s)(a\otimes b)$ by definition. On the other hand,
since $a\in A_\chi$:
\begin{align*}
(\alpha\otimes_G\beta)_s(a\otimes b)
&=\alpha_s(a)\otimes b\\
&=\big(\chi(s)a\big)\otimes b\\
&=\chi(s)(a\otimes b)
\end{align*}
so $\gamma_s(a\otimes b)=(\alpha\otimes_G\beta)_s(a\otimes b)$ for every
$s\in G$ and every $a\otimes b$ in $S_\chi$.
\end{proof}
The following lemma will be useful later:
\begin{lem}\label{l3}
Let $\{A_n\}_{n\in\mathbb Z}$ and $\{B_n\}_{n\in\mathbb Z}$ be saturated
$\mathbb Z$-gradings of $C^*$-algebras $A$ and $B$ and let $\{S_n\}_{n\in
\mathbb Z}$ be the $\mathbb Z$-grading of $A\otimes_\mathbb T B$ as in the
previous propositions. Then $A\otimes_\mathbb
T B$ is generated by $S_1$.
\end{lem}
\begin{proof}
First, we will show that $S_nS_m=S_{n+m}$. We already have that
$S_nS_m\subseteq S_{n+m}$ so it suffices to show the reverse inclusion.
Let $a\otimes b\in S_{n+m}$. Then $a\in A_{n+m}$  and $b\in B_{n+m}$
so $a=\sum a_i a'_i$ and $b=\sum b_i b'_i$ with $a_i\in A_n$,$a'_i\in A_m$,
$b_i\in B_n$, and $b'_i\in B_m$. Therefore $a\otimes b=\sum_{i,j}
(a_i\otimes b_j)(a'_i\otimes b'_j)$ 
where $a_i\otimes b_j\in S_n$ and $a'_i\otimes b'_j
\in S_m$ so $S_{n+m}\subseteq S_nS_m$ and hence $S_nS_m=S_{n+m}$.

Since $1$ generates $\mathbb Z$ as a group, $S_1$ generates $\overline{span}\bigcup
S_n=A\otimes_\mathbb T B$ as a $C^*$-algebra. 
\end{proof}
\section{Ideal Compatibility}
In this section we introduce technical conditions which we will need for
our main result to hold.
\begin{dfn}\label{compatible} 
Let $X$ and $Y$ be correspondences over $C^*$-algebras $A$ and $B$
and let $J_X$, $J_Y$, $J_{X\otimes Y}$ be the Katsura ideals (i.e.
$$J_X=\{a\in A:\phi(a)\in\mathcal K(X)\text{ and }aa'=0
\text{ if }\phi_X(a')=0\}$$ and so on).
We say that $X$ and $Y$ are \emph{``ideal-compatible''} if 
$J_{X\otimes Y}= J_X\otimes J_Y$. 
\end{dfn}

The simplest way for this condition to hold is if the left actions
of $A$ and $B$ on $X$ and $Y$ are injective and implemented by compacts.
In this case it will also be true that the left action of $A\otimes B$
on $X\otimes Y$ will be injective and implemented by compacts. Thus we
will have that $J_X=A$, $J_Y=B$, and $J_{X\otimes Y}=A\otimes B$ so
ideal compatibility is automatic. Thus we have established:
\begin{prop}
Let $X_A$ and $Y_B$ be correpondences such that the left actions of $A$
and $B$ are injective and implemented by compacts. Then $X$ and $Y$ are 
ideal-compatible.
\end{prop}
The following two lemmas are inspired by Lemma 2.6 of \cite{Quigg2}.
\begin{lem}
Let $X$ and $Y$ be correspondences over $A$ and $B$. Then $J_X\otimes J_Y
\subseteq J_{X\otimes Y}$.
\end{lem}
\begin{proof}
Since $\phi_X$ maps $J_X$ injectively into $\mathcal K(X)$
and $\phi_Y$ maps $J_Y$ injectively into $\mathcal K(Y)$, $\phi_{X\otimes
Y}=\phi_X\otimes\phi_Y$ will map $J_X\otimes J_Y$ injectively into
$\mathcal K(X)\otimes K(Y)$, but $\mathcal K(X)\otimes\mathcal K(Y)
=\mathcal K(X\otimes Y)$ so $\phi_{X\otimes Y}$ maps $J_X\otimes J_Y$
injectively into $\mathcal K(X\otimes Y)$. Thus by Proposition 
\ref{CharacterizationOfKatIdeal} $J_X\otimes J_Y\subseteq J_{X\otimes Y}$.
\end{proof}
Recall that if $A$ and $B$ are $C^*$-algebras
and $C$ is an subalgebra of $A$, the triple $(C,A,B)$ is said to 
satisfy the \emph{slice map property} if 
$$C\otimes B=\{x\in A\otimes B:(id_A\otimes \omega)(x)\in C\text{ for all
}\omega\in B^*\}$$
\begin{lem}
Let $X$ and $Y$ be correspondences over $C^*$-algebras $A$ and $B$
and suppose $Y$ is an imprimitivity bimodule.
If $(J_X,A,B)$ satisfies the slice map property, then
$X$ and $Y$ are ideal-compatible.
\end{lem}
\begin{proof}
It suffices to show that $J_{X\otimes Y}\subseteq J_X\otimes J_Y$. Since
$Y$ is an imprimitivity bimodule we have that the left action $\phi_Y$
is an isomorphism $B\cong\mathcal K(Y)$ and thus $\phi_Y$ maps all
of $B$ maps injectively into $\mathcal K(Y)$, so $J_Y=B$.
We must show that
$J_{X\otimes Y}\subseteq J_X\otimes B$. 

Let $c\in J_{X\otimes Y}$. Since $(J_X,A,B)$ satisfies the slice map property,
showing that $c\in J_X\otimes B$ is equivalent to showing that $(id\otimes\omega)
(c)\in J_X$ for all $\omega\in B^*$. Recalling the definition of $J_X$, this
means we must show that $\phi_X\big((id\otimes \omega)(c)\big)\in\mathcal
K(X)$ and that $(id\otimes\omega)(c)a=0$ for all $a\in ker(\phi_X)$.
With this in mind, let $\omega\in B^*$.
Since $\phi_X$ is linear, we have that
\begin{align*}
\phi_X\big((id_A\otimes\omega)(c)\big)&=\big(\phi_X\otimes\omega\big)(c)\\
&=\big(\phi_X\otimes(\omega\circ\phi_Y^{-1}\circ\phi_Y)\big)(c)\\
&=\big(id_{\mathcal K(X)}\otimes(\omega\circ\phi_Y^{-1})\big)\circ
(\phi_X\otimes\phi_Y)(c)\\
&=\big(id_{\mathcal K(X)}\otimes(\omega\circ\phi_Y^{-1})\big)\circ\phi_{X\otimes
Y}(c)
\end{align*}
To see that this is in $\mathcal K(X)$ note that since $c\in J_{X\otimes
Y}$
by assumption, we know that $\phi(c)\in\mathcal
K(X\otimes Y)=\mathcal K(X)\otimes\mathcal K(Y)$. Note that since $Y$ is
an imprimitivity bimodule, $\phi^{-1}_Y$ is well-defined as a map
$\mathcal K(Y)\to B$.
Since $\big(id_{\mathcal
K(X)}\otimes(\omega\circ\phi_Y^{-1})\big)$ maps $\mathcal K(X)\otimes\mathcal
K(Y)\to\mathcal K(X)$ we have that 
$$\big(id_{\mathcal K(X)}\otimes(\omega\circ\phi_Y^{-1})\big)\circ\phi_{X\otimes
Y}(c)\in \mathcal K(X)$$
so $(id\otimes\omega)(c)\in\mathcal K(X)$.

Next, let $a\in ker(\phi_X)$ and factor $\omega$ as $b\cdot\omega'$
for some $b\in B$ and $\omega'\in B^*$ (where $(b\cdot\omega')(b')=\omega'(b'b)$).
Then
\begin{align*}
(id\otimes\omega)(c)a&=(id\otimes\omega)\big(c(a\otimes 1)\big)\\
&=(id\otimes b\cdot\omega')\big(c(a\otimes1)\big)\\
&=(id\otimes\omega')\big(c(a\otimes 1)(1\otimes b)\big)\\
&=(id\otimes\omega')\big(c(a\otimes b)\big)
\end{align*}
but 
\begin{align*}
a\otimes b&\in ker(\phi_X)\otimes B\\
&\subseteq ker(\phi_X\otimes\phi_Y)\\
&=ker(\phi_{X\otimes Y})
\end{align*}
Therefore, since $c\in J_{X\otimes Y}$ we must have $c(a\otimes b)=0$ and
hence $(id\otimes\omega)(c)a=(id\otimes\omega')\big(c(a\otimes b)\big)=0$.
Thus, we have established  that $(id\otimes\omega)(c)\in J_X$ for any $c\in
J_{X\otimes Y}$ and $\omega\in B^*$ so by the slice map property we have
that $J_{X\otimes Y}\subseteq J_X\otimes B=J_X\otimes J_Y$ and thus
(by the previous lemma) $J_{X\otimes Y}=J_X\otimes J_Y$.
\end{proof}
In Example 8.13 of \cite{GraphAlgebraBook},
it is shown that if $E$ is a discrete graph, then
$$J_{X(E)}=\overline{span}\{\delta_v:0<|r^{-1}(v)|<\infty\}$$
where $X(E)$ is the associated correspondence and $\delta_v\in c_0(E^0)$
denotes the characteristic function of the vertex $v\in E^0$. With this in
mind, we give the following proposition:
\begin{prop}\label{GraphCompatible}
Let $E$ and $F$ be discrete graphs and let $X=X(E)$ and $Y=X(F)$ be
the associated correspondences. Then $X$ and $Y$ are ideal compatible.
\end{prop}
\begin{proof} 
Recall that $X\otimes Y=X(E\times F)$. Thus
$$J_{X\otimes Y}=\overline{span}\{\delta_{v\times
w}:0<|r_{E\times F}^{-1}(v\times w)|<\infty\}$$
By definition, $r_{E\times F}=r_E\times r_F$ so $r_{E\times F}^{-1}
(v\times w)=r^{-1}_E(v)\times r^{-1}_F(w)$ and thus $|r^{-1}_{E\times F}(v\times
w)|=|r_E^{-1}(v)|\cdot|r^{-1}_F(w)|$ but $0<|r_E^{-1}(v)|\cdot|r^{-1}_F(w)|<
\infty$ if and only if $0<|r_E^{-1}(v)|<\infty$ and $0<|r^{-1}_F(w)|<\infty$.
Thus we have that
$$J_{X\otimes Y}=\overline{span}\{\delta_{v\times
w}:0<|r_E^{-1}(v)|,|r^{-1}_F(w)|<\infty\}$$
Since $\delta_{v\times w}=\delta_v\delta_w$, if we identify $c_0(E^0\times
F^0)$ with $c_0(E^0)\otimes c_0(F^0)$ in the standard way, we see that 
$\delta_{v\times x}=\delta_v\otimes\delta_w$. Thus
\begin{align*}
J_{X\otimes Y}&=\overline{span}\{\delta_v\otimes
\delta_w:0<|r_E^{-1}(v)|,|r^{-1}_F(w)|<\infty\}\\
&=\overline{span}\{f\otimes g:f\in J_X,g\in J_Y\}\\
&=J_X\otimes J_Y
\end{align*}
Therefore, $X$ and $Y$ are ideal-compatible.
\end{proof}

\begin{dfn}
Let $X$ be a correspondence over a $C^*$ algebra $A$. We will call
this correspondence \emph{Katsura nondegenerate} if $X\cdot J_X=X$.
\end{dfn}

\begin{ex}
Let $X$ be a correspondence over a $C^*$-algebra $A$ such that the
left action is injective and implemented by compacts. In this case
we have that $J_X=A$. Thus:
\begin{align*}
X\cdot J_X&=X\cdot A\\
&=X
\end{align*}
\end{ex}
\begin{dfn}
Recall that a vertex in a directed graph is called a \emph{source}
if it receives no edges. We will call such a vertex a \emph{proper
source} if it emits at least one edge.
\end{dfn}
\begin{prop}
Let $E$ be a directed graph. Then $X(E)$ is Katsura nondegenerate if
and only if $E$ has no proper sources and no infinite receiver emits an edge.
\end{prop}
\begin{proof}
Suppose there is $v\in E^0$ such that $|r^{-1}(v)|=\infty$ and
$|s^{-1}(v)|>0$. Then for every $f\in J_X$ we have $f(v)=0$. Thus
for any $g\in C_c(E^1)$, $f\in J_X$, and $e\in s^{-1}(v)$, 
we have $(g\cdot f)(e)
=g(e)f\big(s(e)\big)=g(e)f(v)=0$. Thus $h(e)=0$ for all $h\in C_c(E^1)
\cdot J_X$ and, taking the limit, $x(e)=0$ for all $x\in X\cdot J_X$. 
Thus $\delta_e\notin X\cdot J_X$ since $\delta_e(e)=1\neq
0$ but $\delta_e\in X$. Therefore $X\neq X\cdot J_X$, i.e. $X$ is not
Katsura nondegenerate.

Similarly, suppose that $E$ has a proper source $v$. Then, since $|r^{-1}(v)|=0$
we must have $f(v)=0$ for all $f\in J_X$. Then for any $g\in C_c(E^1)$ and
$e\in s^{-1}(v)$ we have that $(g\cdot f)(e)=g(e)f(v)=0$ for $f\in J_X$.
Thus by similar reasoning as above we have that $x(e)=0$ for all $x\in
X\cdot J_X$ and so $\delta_e\notin X\cdot J_X$ but $\delta_e\in X$ and
we can again conclude that $X\neq X\cdot J_X$ so $X$ is not Katsura
nondegenerate.

On the other hand, suppose $E$ has no proper sources and no infinite receiver
in $E$ emits an edge.
Let $e\in E^1$ and let $v=s(e)$. Then $|r^{-1}(v)|<\infty$ 
and $|r^{-1}(v)|>0$ by assumption, so function in $J_X$ can be 
supported on $v$. In particular, $\delta_v\in J_X$. Since $\delta_e
\cdot\delta_v=\delta_e$ we know that $\delta_e\in X\cdot J_X$.
Since $e$ was arbitrary, we have that all such characteristic functions
are contained in $X\cdot J_X$. But these functions densely span $C_c(E^1)$
and thus densely span $X$,  so we have that $X\subseteq X\cdot J_X$ and therefore
$X=X\cdot J_X$ so $X$ is Katsura nondegenerate.
 
\end{proof}

\section{Main Result}
We will begin with a few lemmas:
\begin{lem} 
Let $X$ and $Y$ be correspondences over $C^*$-algebras
$A$ and $B$ respectively. Suppose $(\pi_X,\psi_X)$ and $(\pi_Y,\psi_Y)$
are Toeplitz representations of $X$ and $Y$ in $C^*$-algebras $C$ and $D$.
Let $\pi:=\pi_X\otimes\pi_Y$ and $\psi:=\psi_X\otimes\psi_Y$
Then $(\pi,\psi)$ is a Toeplitz representation of $X\otimes Y$ in 
$C\otimes D$.
\end{lem}
\begin{proof}
This follows from the following computations:
\begin{align*}
\psi\big((x\otimes y)\cdot (a\otimes b)\big)
&=\psi_X(x\cdot a)\otimes\psi_Y(y\cdot b)\\
&=\psi_X(x)\pi_X(a)\otimes\psi_Y(y)\pi_Y(b)\\
&=\psi(x\otimes y)\pi(a\otimes b)
\end{align*}
\begin{align*}
\psi\big((a\otimes b)\cdot (x\otimes y)\big)
&=\psi_X(a\cdot x)\otimes\psi_Y(b\cdot y)\\
&=\pi_X(a)\psi_X(x)\otimes\pi_Y(b)\psi_Y(y)\\
&= \pi(a\otimes b)\psi(x\otimes y)
\end{align*}

\begin{align*}
\psi(x\otimes y)^*\psi(x'\otimes y') &= \psi_X(x)^*\psi_X(x')\otimes\psi_Y
(y)^*\psi_Y(y')\\ 
&= \pi_X\big(\langle x, x'\rangle_A\big)\otimes\pi_Y
\big(\langle y,y'\rangle_B\big)\\
&=\pi\big(\langle x,x'\rangle_A\otimes\langle y,y'\rangle_B\big)\\
&=\pi\big(\langle x\otimes y, x'\otimes y'\rangle_{A\otimes B}\big)
\end{align*}
\end{proof}
\begin{lem}
If $(\pi,\psi)$ is the Toeplitz representation of $X\otimes Y$
in the previous lemma, then 
$$\psi^{(1)}(\kappa(S\otimes T))=\psi^{(1)}_X(S)\otimes\psi^{(1)}_Y(T)$$
where $\kappa$ is as in Lemma \ref{TensorOfCompacts}.
\end{lem}
\begin{proof}
Let $x,x'\in X$ and $y,y'\in Y$. Then:
\begin{align*}
\psi^{(1)}(\kappa(\Theta_{x,x'}\otimes\Theta_{y,y'}))
&=\psi^{(1)}(\Theta_{x\otimes y, x'\otimes y'})\\
&=\psi(x\otimes y)
\psi(x'\otimes y')^*\\
&=(\psi_X(x)\otimes\psi_Y(y))(\psi_X(x)\otimes \psi_Y(y))^*\\
&=\psi_X(x)\psi_X(x')^*\otimes\psi_Y(y)\psi_Y(y')^*\\
&=\psi^{(1)}_X(\Theta_{x,x'})\otimes\psi^{(1)}_Y(\Theta_{y,y'})
\end{align*}
Since the rank-ones have dense span in the compacts, this result
extends to any $S\in\mathcal K(X)$ and $T\in\mathcal K(Y)$.
\end{proof}
\begin{lem}\label{CPCovariant}
Let $X$ and $Y$ be ideal-compatible correspondences over $A$ and $B$. 
Then if $(\pi_X,\psi_X)$ and $(\pi_Y,\psi_Y)$ are Cuntz-Pimsner
covariant, then so is $(\pi,\psi)$.
\end{lem}
\begin{proof}
Let $c\in J_{X\otimes Y}$
Since $X$ an $Y$ are ideal-compatible we have that $J_{X\otimes Y}
= J_X\otimes J_Y$ so we can approximate $c$ by a finite sum 
$\sum_i a_i\otimes b_i$ with $a_i\in J_X$ and $b_i\in J_Y$ for each $i$.
 
Thus:
\begin{align*}
\psi^{(1)}(\phi(c))&\approx\psi^{(1)}\left(\phi\left(\sum_ia_i\otimes b_i\right)
\right)\\
&=\psi^{(1)}\left(\sum_i\phi(a_i\otimes b_i)\right)\\
&=\psi^{(1)}\left(\sum_i\kappa\big(\phi_X(a_i)\otimes \phi_Y(b_i)\big)\right)\\
&=\sum_i\psi^{(1)}_X\big(\phi_X(a_i)\big)\otimes\psi^{(1)}_Y\big(\phi_Y
(b_i)\big)\\
&=\sum_i\pi_X(a_i)\otimes\pi_Y(b_i)\\
&=\sum_i\pi(a_i\otimes b_i)\\
&\approx\pi(c)
\end{align*}
where we have used the Cuntz-Pimsner covariance of $(\pi_X,\psi_X)$ and
$(\pi_Y,\psi_Y)$. This establishes that
$(\pi,\psi)$ is Cuntz-Pimsner covariant.
\end{proof}
We are now ready to prove the main result of this paper:
\begin{thm} Let $X$ and $Y$ be ideal-compatible correspondences over 
$C^*$-algebras $A$ and $B$. 
Then $\mathcal{O}_{X\otimes Y}$ can be faithfully imbedded in $\mathcal{O}_X
\otimes_\mathbb T\mathcal{O}_Y$. If $X$ and $Y$ are Katsura nondegenerate,
then $\mathcal O_{X\otimes Y}\cong\mathcal O_X\otimes_\mathbb T\mathcal O_Y$.
\begin{proof}
We will begin by showing the existence of a homomorphism 
$\mathcal{O}_{X\otimes Y}\to \mathcal{O}_X\otimes\mathcal{O}_Y$. 
To show this, we will construct a Cuntz-Pimsner covariant representation
of $X\otimes Y$ in $\mathcal{O}_X\otimes\mathcal{O}_Y$ and then apply
the universal property of Cuntz-Pimsner algebras.

Let $(k_X, k_A)$ and $(k_Y, k_A)$ be the Cuntz-Pimsner covariant
representations of $X$ and $Y$ in $\mathcal{O}_X$ and $\mathcal{O}_Y$ respectively.
Let $\psi=k_X\otimes k_Y$ and $\pi=k_A\otimes k_B$. Then by 
Lemma \ref{CPCovariant},
$(\psi,\pi)$ is Cuntz-Pimsner covariant and so we have a homomorphism
$F:\mathcal O_{X\otimes Y}\to\mathcal O_X\otimes\mathcal O_Y$
such that $$(\psi,\pi)=F\circ(k_{X\otimes Y},k_{A\otimes B})$$
In particular, 
\begin{align}
F(k_{A\otimes B}(A\otimes B))&=\pi(A\otimes B)=
(k_A\otimes k_B)(A\otimes B)\\
F(k_{X\otimes Y}(X\otimes Y))&=\psi(X\otimes Y)=(k_X\otimes
k_Y)(X\otimes Y)
\end{align} 
Let $\{\mathcal O_X^n\}_{n\in\mathbb Z}$ and $\{\mathcal O_Y^n\}_
{n\in\mathbb Z}$ denote the $\mathbb Z$-gradings of $\mathcal O_X$
and $\mathcal O_Y$ associated to the standard gauge actions $\gamma_X$ and
$\gamma_Y$. Then by Proposition~\ref{l1}, the subspaces 
$$S_n:=\{x\otimes y:x\in\mathcal O_X^n, y\in\mathcal O_Y^n\}$$
give a $\mathbb Z$-grading of $\mathcal O_X\otimes_\mathbb T \mathcal O_Y$.
Since (1) shows that $\pi(A\otimes B)\subseteq S_0$ and (2) shows that
$\psi(X\otimes Y)\subseteq S_1$, we can see that $C^*(\psi,\pi)\subseteq
\mathcal O_X\otimes_\mathbb T \mathcal O_Y$ and 
that the action of $\mathbb T$
on $\mathcal O_X\otimes_\mathbb T \mathcal O_Y$ guaranteed 
by Lemma \ref{ActionFromGrading} is a gauge action:
\begin{align*}
\gamma_z(\pi(c))&=\pi(c) & c&\in A\otimes B\\
\gamma_z(\psi(w))&=z\psi(w) & w&\in X\otimes Y
\end{align*}
Also, since $k_A, k_B,k_X$ and $k_Y$ are injective, 
$\pi=k_A\otimes k_B$ and $\psi=k_X\otimes k_Y$ are injective
too. Hence by the gauge invariant uniqueness theorem, $F$ is injective. Thus we have established the first part of the theorem.

Now suppose $X$ and $Y$ are Katsura nondegenerate. We will show
that $(\psi,\pi)$ generates $\mathcal O_X\otimes_\mathbb T\mathcal O_Y$
by showing that $(\psi,\pi)$ generates $S_1$ and applying Lemma \ref{l3}.
Since $\mathcal O_X^1$ is densely spanned
by elements of the form:
$k_X^{n+1}(x) k^n_X(x')^*$
and $\mathcal O_Y^1$ is densely spanned by elements of the form
$k_Y^{n+1}(y)k^n_Y(y')^*$
we have that $S_1$ is densely spanned by elements of the form
\begin{align}
k^{n+1}_X(x)k^n_X(x')^*\otimes k_Y^{m+1}(y) k^m_Y(y')^*
\end{align} 
By symmetry, we may assume $m=n+l$ for some nonnegative $l$.
Then we may assume that
$y=y_1\otimes y_2$ and $y'=y'_1\otimes y'_2$
with $y_1,y'_1\in Y^{\otimes n}$ and $y_2,y'_2\in Y^{\otimes l}$.
Further, since $X$ is Katsura nondegenerate, we can factor $x=x_0a$
and $x'=x'_0a'$ with $x_0,x'_0\in X$ and $a,a'\in J_X$.
Now we can factor (3) as follows:
\begin{align*}
k&^{n+1}_X(x)k^n_X(x')^*\otimes k_Y^{m+1}(y) k^m_Y(y')^*\\
&=\big(k^{n+1}_X(x)\otimes k^{m+1}_Y(y)\big)\big(k^n_X(x')^*\otimes k^m_Y
(y')^*\big)\\
&=\big(k^{n+1}_X(x_0)k_A(a)\otimes k^{n+1}_Y(y_1)k_Y^l(y_2)\big)
\big(k_A(a')^*k^n_X(x'_0)^*
\otimes k^l_Y(y'_2)\big)^*k^n_Y(y'_1)^*\big)\\
&=\big(k^{n+1}_X(x_0)\otimes k^{n+1}_Y(y_1)\big)\big(k_A(a)\otimes k^l_Y(y_2)\big)
\big(k_A(a')^*\otimes k^l_Y(y'_2)^*\big)\big(k^n_X(x'_0)^*\otimes k^n_Y
(y'_1)^*\big)\\
&=\big(k^{n+1}_X(x_0)\otimes k^{n+1}_Y(y_1)\big)\big(k_A(aa'^*)\otimes k^l_Y(y_2)
k^l_Y(y'_2)^*\big)\big(k^n_X(x'_0)^*\otimes k^n_Y
(y'_1)^*\big)\\
&=\big(k^{n+1}_X(x_0)\otimes k^{n+1}_Y(y_1)\big)\big(k_X^{(1)}(\phi_X(aa'^*))
\otimes k_Y^{(1)}
(\Theta_{y_2,y'_2})\big)\big(k^n_X(x')^*\otimes k^n_Y(y'_1)^*\big)\\
&=\psi^{n+1}(x_0\otimes y_1)\big(\psi^{(1)}(\phi_X(aa')\otimes\Theta_{y_2,y'_2})\big)
\psi^n(x'_0\otimes y'_1)^*
\end{align*} 
Since $\psi^{n+1}(x_0\otimes y_1)$, $\psi^{(1)}(\phi_X(aa'^*)\otimes\Theta_{y_2,y'_2})$,
and $\psi^n(x'_0\otimes y'_1)$ are in the algebra generated by $(\psi,\pi)$,
we now know that $(\psi,\pi)$ generates
$S_1$ and so by Lemma~\ref{l3} $(\psi,\pi)$ generates all of 
$\mathcal O_X\otimes_\mathbb T\mathcal O_Y$. 
Therefore $F$ is surjective hence an isomorphism
$\mathcal O_{X\otimes Y}\cong\mathcal O_X\otimes_\mathbb T\mathcal O_Y$.
\end{proof}
\end{thm}
\section{Examples}
We will now give some examples:
\begin{ex}
Let $(A,\mathbb Z,\alpha)$ and $(B,\mathbb Z,\beta)$ be $C^*$-dynamical systems.
Let $X$ be the $C^*$-correspondence $A_A$ with left action given by
$a\cdot x=\alpha_1(a)x$ and let $Y$ be the correspondence $B_B$ with left
action given by $b\cdot y=\beta_1(b)y$. This action is injective and implemented
by compacts and we have that $\mathcal O_X\cong A\rtimes_\alpha
\mathbb Z$ and $\mathcal O_Y\cong B\rtimes_\beta \mathbb Z$ by isomorphisms
which carry the gauge action of $\mathbb T$ to the dual action of $\mathbb
T$ (see \cite{Pimsner}).

Consider the external tensor product $X\otimes Y$. As a right Hilbert
module this is $A_A\otimes B_B$, it carries a right action of $A\otimes B$
characterized by $(x\otimes y)\cdot(a\otimes b)=x\cdot a\otimes y\cdot b$
but since the right actions on $X$ and $Y$ are given by multiplication in
$A$ and $B$, this action of $A\otimes B$ on $X\otimes Y$ is just multiplication
in $A\otimes B$. Further, the $A\otimes B$ valued inner product on $X\otimes
Y$ is given by 
\begin{align*}
\langle x\otimes y,x'\otimes y'\rangle_{A\otimes B}&=\langle x,x'\rangle_A
\otimes\langle y,y'\rangle_B\\
&=x^*x'\otimes y^*y'\\
&=(x\otimes y)^*(x'\otimes y')
\end{align*}
but this is precisely the inner product on $(A\otimes B)_
{A\otimes B}$. Thus $id_A\otimes id_B$ gives a right Hilbert
module isomorphism $X\otimes Y\cong(A\otimes B)_{A\otimes B}$.
The left action of $A\otimes B$ on $X\otimes Y$ will be the tensor
product of the action of $A$ on $X$ and the action of $B$ on $Y$.
Thus 
\begin{align*}
(a\otimes b)\cdot(x\otimes y)&=\alpha_1(a)x\otimes\beta_1(b)y\\
&=\big(\alpha_1(a)\otimes\beta_1(b)\big)(x\otimes y)
\end{align*}
Thus, as an $A\otimes B$ correspondence, $X\otimes Y$ can be
identified with the correspondence associated to the automorphism
$\alpha_1\otimes\beta_1$ on $A\otimes B$. Since $(\alpha_1\otimes
\beta_1)^{\circ n}=(\alpha_n\otimes\beta_n)$ and $(\alpha_1\otimes
\beta_1)^{-1}=(\alpha_{-1}\otimes\beta_{-1})$, the action of $\mathbb
Z$ generated by $\alpha_1\otimes\beta_1$ will be the diagonal action
$\alpha\otimes\beta$ of $\mathbb Z$ on $A\otimes B$. Thus we have that
$\mathcal O_{X\otimes Y}\cong(A\otimes B)\rtimes_{\alpha\otimes\beta}
\mathbb Z$.

Therefore, in this context our main theorem says that
$$(A\otimes B)\rtimes_{\alpha\otimes\beta}\mathbb Z\cong
(A\rtimes_\alpha\mathbb Z)\otimes_\mathbb T(B\rtimes_\beta
\mathbb Z)$$
In later work, we hope to investigate whether this result generalizes
to groups other than $\mathbb Z$.
\end{ex}
\begin{ex} [Products of Topological Graphs]
Let $E=(E^0,E^1,r,s)$ and $F=(F^0,F^1,r',s')$ be source-free 
topological graphs with $r$ and $r'$ proper. 
Then the left actions of $X(E)$ and $X(F)$ will be
injective and implemented by compacts. Recall from Example \ref{TensorTopGraph}
that $X(E)\otimes X(F)\cong X(E\times F)$ where $E\times F$ is the product
graph. Our main result says that $\mathcal O_{X(E\times F)}\cong
\mathcal O_{X(E)}\otimes_\mathbb T\mathcal O_{X(F)}$ which translates
to $$C^*(E\times F)\cong C^*(E)\otimes_\mathbb T C^*(F)$$ Note that if $E$
and
$F$ are discrete graphs, this coincides with Kumjian's result in \cite{Kum}.
\end{ex}
\begin{ex}(Products of Discrete Graphs)
Let $E$ and $F$ be discrete graphs with no proper sources and such that
no infinite receiver emits an edge. From the discussion in Section 4 we know
that the graph correspondences $X(E)$ and $X(F)$
are ideal compatible and Katsura nondegenerate. By the same reasoning as
in the previous example we have that $$C^*(E\times F)\cong C^*(E)\otimes_\mathbb
T C^*(F)$$
Note that this stronger than the result in \cite{Kum} where the graphs
are required to be source-free and row-finite.
\end{ex}
\begin{ex}
Let $A$ and $B$ be $C^*$-algebras, and let 
$X$ be a correspondence over $A$. Viewing $B$ as the correspondence
$\tensor[_B]B{_B}$ we can form the $A\otimes B$ correspondence $X\otimes
B$. Suppose that $X$ and $B$ are ideal compatible and Katsura nondegenerate
(in fact $B$ will automatically be Katsura nondegenerate).
Recall that $\mathcal O_B\cong B\otimes C(\mathbb T)$ with gauge action
$\iota\otimes\lambda$ where $\iota$ is the trivial action and $\lambda$ is
left translation. Thus our main result says that $\mathcal O_{X\otimes B}
\cong\mathcal O_X\otimes_\mathbb T (B\otimes C(\mathbb T))$. Identifying
$C(\mathbb T)$ with $C^*(\mathbb Z)$ we have $\mathcal O_{X\otimes B}
\cong\mathcal O_X\otimes_\mathbb T (B\otimes C^*(\mathbb Z))$ and characterizing
the $\mathbb T$ balanced tensor product in terms of the $\mathbb Z$-gradings
as we have been, we see that
$$\mathcal O_{X\otimes B}\cong\overline{span}\{x\otimes b\otimes w\in
\mathcal O_X^n\otimes B\otimes C^*(\mathbb Z)^n:n\in\mathbb Z\}$$
But since $C^*(\mathbb Z)^n=span(u_n)$ (where $u_n$ denotes the unitary
in $C^*(Z)$ associated to $n$) we can rephrase this as
\begin{align}
\overline{span}\{x\otimes b\otimes u_n\in
\mathcal O_X\otimes B\otimes C^*(\mathbb Z):x\in\mathcal O_X^n\}\label{set1}
\end{align}
Now, let $\gamma$ denote the gauge action of $\mathbb T$ on $\mathcal O_X$
and let $\delta^\gamma$ be the dual coaction of $\mathbb Z$. Recall that
this coaction can be characterized by the property that $\delta^\gamma(x)
=x\otimes u_n$ whenever $x\in \mathcal O_X^n$. Since the subspaces
$\mathcal O_X^n$ densely span $\mathcal O_X$, their images under
$\delta^\gamma$ will densely span $\delta^\gamma(\mathcal O_X)$.
Therefore: 
\begin{align*}
\delta^\gamma(\mathcal O_X)&=\overline{span}\{x\otimes u_n:x\in \mathcal
O_X^n\}\\
(\delta^\gamma\otimes id_B)(\mathcal O_X\otimes B)&=\overline{span}\{x\otimes
u_n\otimes b:x\in \mathcal O_X^n, b\in B\}\\
\sigma_{23}\circ(\delta^\gamma\otimes id_B)(\mathcal O_X\otimes B)&=\overline{span}\{x\otimes
b\otimes u_n:x\in \mathcal O_X^n, b\in B\}
\end{align*}
Noting that
$id_B$, and $\sigma_{23}$ are isomorphisms (where $\sigma_{23}$ is the
map which exchanges the second and third tensor factors) and $\delta^\gamma$
is an injective $*$-homomorphism, we see that 
$\sigma_{23}\circ(\delta^\gamma\otimes id_B)(\mathcal O_X\otimes B)$ is an
injective $*$-homomorphism and is thus an isomorphism onto its image.
But its image is $\overline{span}\{x\otimes
b\otimes u_n:x\in \mathcal O_X^n, b\in B\}$ and by (\ref{set1}) this
is isomorphic to $\mathcal O_{X\otimes B}$. Therefore we have shown that:
$$\mathcal O_{X\otimes B}\cong\mathcal O_X\otimes B$$
This result is already known, and was used in \cite{Quigg2} to prove
facts about coactions on Cuntz-Pimsner algebras.
\end{ex}
\begin{ex}
Given a $C^*$-algebra $A$ and a completely positive map $\Phi$, Kwasniewski
defines (\cite{KWASNIEWSKI} Definition 3.5) a crossed product of $A$ by $\Phi$
denoted by $C^*(A,\Phi)$. In Theorem 3.13 of \cite{KWASNIEWSKI}, it is shown
that $C^*(A,\Phi)\cong\mathcal O_{X_\Phi}$ where $X_\Phi$ is the correspondence
associated with $\Phi$ as in Definition \ref{CPCor}. If $\Phi$ is an endomorphism
this reduces to the Exel crossed product \cite{Exel}. 

Suppose $A$ and $B$ are $C^*$-algebras and $\Phi:A\to A$ and $\Psi:B\to B$
are completely positive maps. Furthermore, suppose that the associated
correspondences $X_\Phi$ and $X_\Psi$ are ideal compatible and Katsura 
nondegenerate. Then our main
result states that $\mathcal O_{X_\Phi\otimes X_\Psi}\cong \mathcal O_{X_\Phi}
\otimes_\mathbb T \mathcal O_{X_\Psi}$. Recalling that $X_\Phi\otimes X_\Psi
\cong X_{\Phi\otimes\Psi}$ and using the crossed product notation, we get
$$C^*(A\otimes B,\Phi\otimes\Psi)\cong C^*(A,\Phi)\otimes_\mathbb TC^*(B,\Psi)$$
\end{ex}

\providecommand{\bysame}{\leavevmode\hbox to3em{\hrulefill}\thinspace}
\providecommand{\MR}{\relax\ifhmode\unskip\space\fi MR }
\providecommand{\MRhref}[2]{%
  \href{http://www.ams.org/mathscinet-getitem?mr=#1}{#2}
}
\providecommand{\href}[2]{#2}


\end{document}